\pdfoutput=1

\documentclass[12pt,reqno]{amsart}
\usepackage{amsmath}
\usepackage{amsthm}
\usepackage{amssymb}
\usepackage{amsfonts}
\usepackage{mathrsfs}
\usepackage[backrefs]{amsrefs}
\usepackage{mathtools}

\title{Hamiltonian Groups with Perfect Order Classes}
\author{James McCarron}
\address{Maplesoft\\
615 Kumpf Drive,\\
Waterloo, ON, Canada\\
N2V 1K8}
\email{\texttt{james@maplesoft.com}}
\date{\today}

\swapnumbers
\theoremstyle{plain}
\newtheorem{theorem}{Theorem}[section]
\newtheorem{proposition}[theorem]{Proposition}
\newtheorem{corollary}[theorem]{Corollary}
\newtheorem{lemma}[theorem]{Lemma}

\newtheorem*{theorem*}{Theorem} 
\newtheorem*{maintheorem}{Main Theorem}

\numberwithin{theorem}{section}

\newcommand{\defn}[1]{\textbf{#1}}

\newcommand{\abs}[1]{\ensuremath{\left\lvert #1\right\rvert}}
\newcommand{\order}[1]{\ensuremath{\abs{#1}}} 
\newcommand{\iso}{\ensuremath{\simeq}}

\newcommand{\sym}[1]{\ensuremath{S_{#1}}}
\newcommand{\alt}[1]{\ensuremath{A_{#1}}}

\newcommand{\cyclic}[1]{\ensuremath{\operatorname{C}_{#1}}}
\newcommand{\quat}{\ensuremath{Q}}

\DeclareMathOperator{\eulerphi}{\phi}

\let\divides\mid

\let\congruent\equiv

\newcommand{\ordclass}[2]{\ensuremath{{#1}^{[{#2}]}}}
\newcommand{\ordcount}[2]{\ensuremath{f_{#1}(#2)}}

\begin{document}

\begin{abstract}
A finite group is said to have \textit{perfect order classes}
if the number of elements of any given order is either
zero or a divisor of the order of the group.
A group is \textit{Hamiltonian} if it is non-abelian and every
one of its subgroups is normal.

The purpose of this note is to describe precisely the
finite Hamiltonian groups with perfect order classes.
We show that a finite Hamiltonian group has perfect order
classes if, and only if, it is isomorphic to the direct
product of the quaternion group of order $8$, a non-trivial
cyclic $3$-group and a group of order at most $2$.

\begin{theorem*}
A finite Hamiltonian group $G$ has perfect order classes if,
and only if, it is isomorphic either to $\quat\times\cyclic{3^k}$
or to $\quat\times\cyclic{2}\times\cyclic{3^k}$,
for some positive integer $k$.
\end{theorem*}
\end{abstract}

\maketitle

\section{Introduction}

Let $G$ be a finite group,
and define an equivalence relation on the elements of $G$
by declaring that elements $a$ and $b$ of $G$ are equivalent
if, and only if, they have the same order.
The equivalence class of an element $g$ in $G$
is called its \defn{order class} and is denoted by
$\ordclass{g}{G}$.
Thus,
\begin{displaymath}
	\ordclass{g}{G} = \{ x\in G : \order{x} = \order{g} \},
\end{displaymath}
The cardinality of $\ordclass{g}{G}$ is denoted by $\ordcount{g}{G}$.
For a positive integer $k$, we define $\ordcount{k}{G}$
to be the number of elements of order $k$ in $G$.
It is equal to zero if $G$ has no elements of order $k$.
In any group $G$, the order class $\ordclass{1}{G}$ of the identity
element is always a singleton, so $\ordcount{1}{G} = 1$ however the
subscript is interpreted.

We say that $G$ has \defn{perfect order classes} if $\ordcount{g}{G}$
is a divisor of the order of $G$, for each element $g$ in $G$.
Groups with this property seem first to have been investigated
in~\cite{FinchJones2002}.
See also~\cites{FinchJones2003,JonesToppin2011}.

The symmetric group $\sym{3}$ of degree $3$ has perfect order classes
because it has $2$ elements of order $3$, and $2$ is a divisor of the
order $6$ of $\sym{3}$, and also has $3$ elements of order $2$,
and $3$ is again a divisor of $6$.
The alternating group $\alt{4}$ of degree $4$ however, does not have
perfect order classes, as it has a total of $8$ elements of order $3$, and $8$
is not a divisor of the order $12$ of $\alt{4}$.

A group is \defn{Hamiltonian} if it is non-abelian and if each of its 
subgroups is normal.

The object of this note is to describe Hamiltonian groups with
perfect order classes.
We prove the following theorem, in which $\cyclic{n}$ denotes
the cyclic group of order $n$ and $\quat$ denotes the quaternion
group of order $8$.

\begin{maintheorem}\label{thm:main}
A finite Hamiltonian group has perfect order classes if,
and only if, it has the form
$\quat\times\cyclic{3^k}$ or the form $\quat\times\cyclic{2}\times\cyclic{3^k}$,
for some positive integer $k$.
\end{maintheorem}

It follows from this that the smallest Hamiltonian group with
perfect order classes is the group $\quat\times\cyclic{3}$ of order $24$.
In particular, despite being involved in every Hamiltonian group,
the quaternion group $\quat$ of order $8$ does not itself have perfect
order classes.
(It has a total of $6$ elements of order $4$, and $6$ does not divide $8$.)

\section{Preliminaries}

All groups are supposed to be finite, and will be written multiplicatively.
We use $1$ to denote the trivial group of order $1$,
as well as to denote the identity element of any group.

We denote the quaternion group of order $8$ by $\quat$ and,
for a positive integer $n$, the cyclic group of order $n$
is denoted by $\cyclic{n}$.
If $G$ is any group, and $n$ a non-negative integer,
then $G^n$ denotes the direct product of $n$ copies of $G$
(with the understanding that $G^0$ is the trivial group).

For a positive integer $n$, Euler's function $\eulerphi(n)$ counts the
number of positive integers less than and relatively prime to $n$.
Recall that it is a multiplicative function:
$\eulerphi(ab) = \eulerphi(a)\eulerphi(b)$,
for relatively prime positive integers $a$ and $b$.
For a prime number $p$ and a positive integer $n$, we have
\begin{displaymath}
\eulerphi(p^n) = p^{n-1}(p-1),
\end{displaymath}
a formula that we use frequently.
Especially, we have $\eulerphi(p) = p - 1$.

The next result is a key result in the study of groups with perfect order classes.
It follows from the observation that each order class $\ordclass{g}{G}$
in a finite group is a disjoint union of equivalence classes
for the equivalence relation that identifies two group elements
if they generate the same cyclic subgroup,
and each equivalence class of this finer relation has cardinality
equal to $\eulerphi(\order{g})$.

\begin{proposition}\cite{Das2009b}*{Proposition 2.2}
If $g$ is an element of a finite group $G$,
then $\ordcount{g}{G}$ is divisible by $\eulerphi(\order{g})$.
\end{proposition}

\begin{corollary}\cite{FinchJones2002}*{Proposition 1, Corollary 1}
If $G$ is a finite group with perfect order classes then,
for each prime divisor $p$ of $G$, the order of $G$ is divisible
by $p - 1$.
In particular, every non-trivial finite group with perfect order classes has even order.
\end{corollary}

The structure of Hamiltonian groups is known.
\begin{theorem}\cite{Robinson}*{5.3.7}\label{thm:hamstruct}
A finite Hamiltonian group $G$ is a direct product of $\quat$,
an elementary abelian $2$-group, and an abelian
group of odd order.
\end{theorem}
Either or both of the elementary abelian direct factor and the
direct factor of odd order may be trivial, of course, as $\quat$
is itself Hamiltonian.

We make some simple remarks about order classes and their lengths.
\begin{lemma}\label{lem:hallorderclosed}
A normal Hall subgroup of a finite group contains the complete
order class of each of its members.
\end{lemma}
\begin{proof}
Let $H$ be a normal Hall subgroup of a finite group $G$,
and let $h$ be a member of $H$.
Since the order $\order{H}$ and the index $[G:H]$ of $H$
are coprime, there are integers $\alpha$ and $\beta$ for
which $\alpha\order{H} + \beta [G:H] = 1$.

If $g$ is any member of $\ordclass{h}{G}$,
then $\order{g} = \order{h}$ is a divisor of $\order{H}$, so we have
$g = g^{\alpha\order{H}}g^{\beta [G:H]} = g^{\beta [G:H]}$.
Then, in the quotient group $G/H$ of order $[G:H]$ we
compute $g^{\beta [G:H]}H = (gH)^{\beta [G:H]} = H$,
whence $g$ belongs to $H$.
Since $g$ was an arbitrary member of $\ordclass{h}{G}$,
we conclude that $\ordclass{h}{G}\subset H$, as claimed.
\end{proof}

The following lemma is proved in essentially the same way.
\begin{lemma}\label{lem:hallsuck}
If $H$ is a normal Hall subgroup of a finite group $G$ then,
for each divisor $d$ of the order of $H$,
we have $\ordcount{d}{G} = \ordcount{d}{H}$.
\end{lemma}

\begin{corollary}\label{cor:coprimedp}
Let $G = A\times B$, where $A$ and $B$ have relatively prime orders.
If $a$ divides $\order{A}$ and $b$ divides $\order{B}$, then
$\ordcount{ab}{G} = \ordcount{a}{A}\ordcount{b}{B}$.
\end{corollary}

The following well-known result on consecutive prime powers
is a consequence of the Bang-Zsigmondy theorem.
\begin{lemma}\label{lem:conspp}
Let $p$ and $q$ be prime numbers and let $m$ and $n$ be
positive integers such that $p^m - 1 = q^n$.
Then we have one of the following:
\begin{enumerate}
\item{$p = n = 3$ and $q = m = 2$;}
\item{$(n,p) = (1,2)$, and $q = 2^m -1$ is a Mersenne prime (whence $m$ is prime); or,}
\item{$(m,q) = (1,2)$, and $p = 2^n + 1$ is a Fermat prime (whence $n$ is a power of $2$).}
\end{enumerate}
\end{lemma}

\section{Hamiltonian Groups with Perfect Order Classes}

The following useful lemma provides formulae for the numbers
of elements of a given order in a finite Hamiltonian group.
We express these in terms of a direct decomposition of the
form $G = \quat\times E\times A$, as described in Theorem~\ref{thm:hamstruct}.

\begin{lemma}\cite{Tarnauceanu2013}*{Theorem 2.1}\label{lem:ordcount}
Let $G = \quat\times E\times A$ be a finite Hamiltonian group,
where $E\iso\cyclic{2}^e$, for $e$ a non-negative integer, and $A$ is an abelian group of odd order.
For each odd divisor $d$ of the order of $G$, we have
\begin{enumerate}
\item{$\ordcount{d}{G} = \ordcount{d}{A}$;}
\item{$\ordcount{2d}{G} = (2^{e+1} - 1)\cdot \ordcount{d}{A}$; and,}
\item{$\ordcount{4d}{G} = 3\cdot 2^{e+1}\cdot \ordcount{d}{A}$.}
\end{enumerate}
\end{lemma}

Note especially the cases in which $d = 1$:
\begin{displaymath}
\ordcount{2}{G} = 2^{e+1} - 1
\end{displaymath}
and
\begin{displaymath}
\ordcount{4}{G} = 3\cdot 2^{e+1}.
\end{displaymath}

Since a finite Hamiltonian group $G$ has elements of order $4$,
if it is to have perfect order classes, then $\ordcount{4}{G}$,
which is divisible by $3$,
must be a divisor of the order of $G$.

\begin{corollary}
The order of a finite Hamiltonian group with perfect order classes
is a multiple of $3$.
\end{corollary}

We now turn to the proof of the main theorem.
Let us begin by showing that each type of Hamiltonian group
described there has perfect order classes.

\begin{lemma}\label{lem:sufficiency}
If $k$ is a positive integer, then the groups
$G = \quat\times\cyclic{2}\times\cyclic{3^k}$,
and
$G = \quat\times\cyclic{3^k}$,
have perfect order classes.
\end{lemma}
\begin{proof}
Consider first the case $G = \quat\times\cyclic{3^k}$,
and note that
\begin{displaymath}
\order{G} = \order{\quat\times\cyclic{3^k}} = 8\cdot 3^k.
\end{displaymath}
Then $G$ has an unique element of order $2$,
and $6$ elements of order $4$,
so both $\ordcount{2}{G}$ and $\ordcount{4}{G}$ divide $\order{G}$.
For each $j$ with $1\leq j\leq k$,
we have
\begin{displaymath}
\ordcount{3^j}{G} = \ordcount{3^j}{\cyclic{3^k}} = \eulerphi(3^j) = 2\cdot 3^{j-1},
\end{displaymath}
which is a divisor of $\order{G}$.
Next,
\begin{displaymath}
\ordcount{2\cdot 3^j}{G} = 2\cdot \ordcount{3^j}{\cyclic{3^k}} = 2\cdot\eulerphi(3^j) = 4\cdot 3^{j-1},
\end{displaymath}
and
\begin{displaymath}
\ordcount{4\cdot 3^j}{G} = 3\cdot 2\cdot \ordcount{3^j}{\cyclic{3^k}} = 3\cdot 2\cdot\eulerphi(3^j) = 3\cdot 4\cdot 3^{j-1} = 4\cdot 3^j,
\end{displaymath}
and these divide the order of $G$.

Now consider the case $G = \quat\times\cyclic{2}\times\cyclic{3^k}$.
In this case there are $3$ elements of order $2$ and $12$ of order $4$,
while the number of elements of order a power of $3$ is the same as
in the previous case.
The numbers of elements of orders $2\cdot 3^j$ and $4\cdot 3^j$ are twice
as many as the previous case, and these remain divisors of the order of $G$.

This completes the proof.
\end{proof}

The remainder of this section is devoted to proving the converse.
To that end, we suppose that $G$ is a finite Hamiltonian group
with perfect order classes.
Then
\begin{displaymath}
G = \quat\times E\times T\times P,
\end{displaymath}
where $E\iso\cyclic{2}^e$ is an elementary abelian $2$-group of rank $e\geq 0$,
$T$ is a non-trivial abelian $3$-group,
and $P$ is an abelian group of odd order, coprime to $3$.
Note that we have
\begin{displaymath}
\order{G} = 8\cdot 2^e\cdot 3^k\cdot \order{P} = 2^{e+3}\cdot 3^k\order{P},
\end{displaymath}
for some positive integer $k$.

The argument to follow will show, in a sequence of lemmata,
first that the subgroup $P$ has prime-power order,
then that $P$ is actually trivial,
and finally that the $3$-subgroup $T$ is cyclic
and that the elementary abelian $2$-subgroup $E$ is either
trivial or of order $2$.

Let us first show that our subgroup $P$ has prime power order.

\begin{lemma}\label{lem:step1}
The order of $G$ is divisible by at most one prime number greater than $3$.
In particular, $P$ is a $p$-group, for some prime number $p > 3$.
\end{lemma}
\begin{proof}
Suppose, to the contrary, that $G$ is divisible by two odd
primes $p$ and $q$, both greater than $3$.
Then
\begin{displaymath}
\ordcount{12pq}{G} = 3\cdot 2^{e+1} \cdot \ordcount{3}{G} \cdot \ordcount{p}{G}\cdot \ordcount{q}{G}.
\end{displaymath}
Now each of $\ordcount{3}{G}$, $\ordcount{p}{G}$ and $\ordcount{q}{G}$ is even,
so $\ordcount{12pq}{G}$ is divisible by $2^{e+4}$,
which implies that $2^{e+4}$ divides $\order{G}$,
a contradiction.
\end{proof}

From the lemma, it follows that our group $G$ has the form
\begin{displaymath}
G = \quat\times\cyclic{2}^e\times T\times P,
\end{displaymath}
where $P$ is a $p$-group, for some prime $p > 3$.
The next step is to show that $P$ is trivial.

\begin{lemma}\label{lem:step2}
Let $G = \quat\times\cyclic{2}^e\times T\times P$,
where $T$ is a non-trivial abelian $3$-group,
$e\geq 0$,
and $P$ is an abelian $p$-group, for some prime $p > 3$.
If $G$ has perfect order classes,
then $P$ is trivial.
\end{lemma}
\begin{proof}
Suppose that $\order{P} = p^m$, where $m\geq 1$;
we shall derive a contradiction.
Then $G$ has an element of order $p$,
so $\ordcount{p}{G}$ is divisible by $p - 1$.
Suppose that $\order{T} = 3^k$, with $k\geq 1$.
Note the formula
\begin{displaymath}
\order{G} = 2^{e+3}\cdot 3^k\cdot p^m,
\end{displaymath}
and, in particular, the highest powers of $2$ and $3$ that divide
the order of $G$.

Since $\ordcount{3}{G} = \ordcount{3}{T}$, we can write
\begin{displaymath}
\ordcount{3}{G} = 3^{\lambda} - 1,
\end{displaymath}
where $\lambda\geq 1$ is the rank of $T$.

If $\lambda$ is even, then $3^{\lambda}\congruent 1\pmod{8}$,
so $\ordcount{3}{G}$ is divisible by $8$.
Then $\ordcount{12}{G} = 3\cdot 2^{e+1}\cdot\ordcount{3}{G}$
is divisible by $2^{e+4}$.
Since $\ordcount{12}{G}$ divides $\order{G} = 2^{e+3}\cdot 3^k\cdot p^m$,
this is a contradiction.
Therefore, $\lambda$ is odd.

If $\lambda = 1$, then $T\iso\cyclic{3^k}$ is cyclic,
and so has an element of order $3^k$,
and the number of elements of order $3^k$ is equal to
$\eulerphi(3^k) = 2\cdot 3^{k-1}$.
Then from Lemma~\ref{lem:ordcount} the number of elements of $G$
whose order is equal to $4\cdot p\cdot 3^{k}$ is given by
\begin{displaymath}
\ordcount{4\cdot p\cdot 3^k}{G}
= 3\cdot 2^{e+1}\cdot\ordcount{3^k}{T}\ordcount{p}{P}
= 2^{e+2}\cdot 3^k\cdot\ordcount{p}{P}.
\end{displaymath}
Since $\ordcount{4\cdot p\cdot 3^k}{G}$ must divide the order
of $G$, and since $p - 1$ is a divisor of $\ordcount{p}{P}$,
it follows that
\begin{displaymath}
2^{e+2}\cdot 3^k\cdot(p - 1) \divides 2^{e+3}\cdot 3^k \cdot p^m.
\end{displaymath}
It is clear from this divisibility relation that $p - 1$ cannot
be a multiple of $3$, so we can write $p = 6n - 1$,
for some positive integer $n$.
Then $p - 1 = 6n - 2 = 2(3n - 1)$,
and $3n - 1$ is a divisor of $\order{G}$ coprime to both $3$ and $p$.
Consequently, we may write $3n - 1 = 2^{\alpha}$, for some positive integer $\alpha$.
Then $p - 1 = 2^{\alpha + 1}$ and we see that $\order{G} = 2^{e+2}\cdot 3^k\cdot p^m$
is divisible by $2^{e + \alpha + 3}$, which is impossible.
Therefore, we cannot have $\lambda = 1$, and $T$ is not cyclic.

Since $\lambda$ is odd, we must have $\lambda\geq 3$.
Because $\ordcount{3}{G} = 3^{\lambda} - 1$ is coprime to $3$
and is an even divisor of $\order{G}$,
hence, $3^{\lambda} - 1$ divides $2^{e+3}p^m$.
And, since $\lambda\geq 3$, it follows that $3^{\lambda} - 1$
is not a power of $2$ so, in fact, $2p$ divides $3^{\lambda} - 1$.
Since $\lambda$ is odd, hence, $1 + 3 + \cdots + 3^{\lambda - 1}$
is odd so we have,
\begin{displaymath}
\ordcount{3}{G} = 3^{\lambda} - 1 = 2p^{\alpha},
\end{displaymath}
for some integer $\alpha\geq 1$.
In particular, we have
\begin{displaymath}
3^{\lambda}\congruent 1\pmod{p}.
\end{displaymath}
From this, and from Fermat's little theorem,
we obtain that the order $s$ of $3$ modulo $p$ divides $p-1= \eulerphi(p)$ which,
in turn, is a divisor of $\order{G}$;
and also, $s$ divides $\lambda$.
Since $\gcd(p,p-1) = 1$ it follows that $p - 1$,
and hence also its divisor $s$,
is of the form $2^{\alpha}3^{\beta}$,
for suitable integers $\alpha,\beta\geq 0$.
But $s$ divides the odd integer $\lambda$,
so it must be that $s$ is a (positive) power of $3$.
(Note that $s\neq 1$ since $p$ is odd.)
In particular, $3$ divides $\lambda$, so we can write
$\lambda = 3\sigma$, where $\sigma\geq 1$ is an odd integer.
Now,
\begin{displaymath}
\ordcount{3}{G} = 3^{3\sigma} - 1 = 27^{\sigma} - 1 = 26(1 + 27 + \cdots + 27^{\sigma - 1}),
\end{displaymath}
so $13$ divides $\ordcount{3}{G}$, which divides $\order{G}$,
and it follows that $p = 13$.
But then $4$ divides $p - 1 = 12$,
$\ordcount{12p}{G}$
is divisible by $2^{e+4}$,
a final contradiction.

This completes the proof.
\end{proof}

We now know that $G$ is a $\{2,3\}$-group.
The final step is to prove that the Sylow $3$-subgroup $T$
of $G$ is cyclic, and that the rank of $E$ is either $0$ or $1$.
For this, we need the following elementary observation.

\begin{lemma}\label{lem:cons23}
Let $a$ and $b$ be non-negative integers.
If $2^a - 1 = 3^b$, then $(a,b)\in \{ (1,0), (2,1) \}$.
If $3^a - 1 = 2^b$, then $(a,b)\in \{ (1,1), (2,3) \}$.
\end{lemma}
\begin{proof}
Suppose that $2^a - 1 = 3^b$.
If $a = 2s$ is even, then $3^b = (2^s - 1)(2^s + 1)$.
If $2^s - 1 = 1$, then $s = 1$ and $a = 2$, so $3^b = 3$ and $b = 1$.
If $2^s > 1$, then $3$ divides both $2^s - 1$ and $2^s + 1$,
hence, $3$ also divides $(2^s + 1) - (2^s - 1) = 2$,
which is absurd.
If $a$ is odd, then $2^a \equiv 2\pmod{3}$,
so that $3^b = 2^a - 1 \equiv 1\pmod{3}$,
which is possible only if $b = 0$, in which case $a = 1$.

Now suppose that $3^a - 1 = 2^b$.
If $a = 2s$ is even, then $2^b = (3^s - 1)(3^s + 1)$,
so that $3^s - 1$ and $3^s + 1$ are powers of $2$.
Since only one of $3^s - 1$ and $3^s + 1$ can be divisible
by $4$, it follows that $2\in\{3^s - 1, 3^s + 1\}$.
If $3^s + 1 = 2$, then $s = 0$ and $3^s - 1 = 0$,
which is impossible.
Therefore, $3^s - 1 = 2$, so that $s = 1$ and $a = 2$,
and hence, $2^b = 8$, yielding $b = 3$.
If $a$ is odd, then $2^b = 2(1 + 3 + \cdots + 3^{a-1})$,
and $1 + 3 + \cdots 3^{a-1}$ is odd.
Therefore, $1 + 3 + \cdots + 3^{a-1} = 1$,
giving $a = 1$.
Hence $2^b = 2$, and so $b = 1$ also.
\end{proof}

\begin{lemma}\label{lem:step3}
Let $G = \quat\times\cyclic{2}^e\times{T}$,
where $T$ is an abelian group of order $3^k$, for a positive integer $k$.
If $G$ has perfect order classes, then $e\in\{0,1\}$
and $T$ is cyclic.
\end{lemma}
\begin{proof}
Since $\ordcount{2}{G} = 2^{e+1} - 1$ divides the order $\order{G} = 2^{e+3}3^k$
of $G$, and since $\ordcount{2}{G}$ is odd, therefore,
\begin{displaymath}
2^{e+1} - 1 = 3^n,
\end{displaymath}
for some non-negative integer $n$.
According to Lemma~\ref{lem:cons23},
the only solutions are $(e,n) = (0,0)$ and $(e,n) = (1,1)$.
Thus, $e\in\{0,1\}$, as claimed.

Next, the number of elements of order $3$ in $G$ has the form
\begin{displaymath}
\ordcount{3}{G} = 3^{\lambda} - 1,
\end{displaymath}
where $\lambda\geq 1$ is the rank of $T$.
As in the proof of Lemma~\ref{lem:step2}, the rank $\lambda$
must be odd.

Since $\ordcount{3}{G} = 3^{\lambda}-1$ is a divisor of $\order{G}$,
and is even, we must have
\begin{displaymath}
3^{\lambda} - 1 = 2^{\alpha}3^{\beta},
\end{displaymath}
for some integers $\alpha$ and $\beta$ with $\alpha\geq 1$,
and $\beta\geq 0$.
However, $3$ and $3^{\lambda} - 1$ are coprime,
which implies that $\beta = 0$ and $3^{\lambda} - 1  = 2^{\alpha}$.
By Lemma~\ref{lem:cons23},
the only solutions for this are $(\lambda,\alpha) = (1,1)$
and $(\lambda,\alpha) = (2,3)$.
Since $\lambda$ is odd, we must have $\lambda = 1$.
This implies that $T$ is cyclic, completing the proof.
\end{proof}

\begin{proof}[Proof of the Main Theorem]
From Lemma~\ref{lem:sufficiency} we see that Hamiltonian groups of the form
$\quat\times\cyclic{3^k}$ and $\quat\times\cyclic{2}\times\cyclic{3^k}$
have perfect order classes.
The converse follows from Lemmas~\ref{lem:step1},
\ref{lem:step2} and \ref{lem:step3}.
\end{proof}


\end{document}